\documentclass[reqno,10pt,a4paper]{article}
\usepackage{amsmath}
\usepackage{amsthm}
\usepackage{amsfonts}
\usepackage{amssymb}
\usepackage{indentfirst}
\usepackage{marginnote}
\usepackage{xcolor}
\usepackage[left=3cm,top=3cm,right=3cm,bottom=3cm]{geometry}
\usepackage{color}
\usepackage[colorlinks=true, linkcolor=black, citecolor=black]{hyperref}

\newtheorem{thm}{Theorem}[section]
\newtheorem{lem}[thm]{Lemma}
\newtheorem{co}[thm]{Corollary}

\newtheorem{re}[thm]{Remark}
\numberwithin{equation}{section}

\begin{document}
\title{Exact convergence rate in the central limit theorem  for a branching process with immigration in a random environment\\
}
\author{Chunmao Huang$^{a}$\footnote{Corresponding author at: Harbin Institute of Technology (Weihai), Department of Mathematics, 264209, Weihai, China.  \newline \indent \ \ Email addresses: cmhuang@hitwh.edu.cn  (Chunmao Huang).}, \;\;  Rui Zhang$^{a}$, \;\;  Zhiqiang Gao$^{b}$
\\
\small{\emph{$^{a}$Harbin Institute of Technology (Weihai), Department of Mathematics, 264209, Weihai, China}}\\
\small{\emph{$^{b}$Beijing Normal University, School of Mathematical Sciences, 100875, Beijing, China}}}
\maketitle

\begin{abstract}

\bigskip
 Let ($Z_n$) be a branching process with immigration in an independent and identically distributed random environment. Under necessary moment conditions, we show the exact convergence rate in the central limit theorem on $\log Z_n$ by using the convergence rates of the logarithm  of submartingale and  the   result of the corresponding random walk on the Berry-Esseen bound.\\*



\end{abstract}
\section{Introduction}\label{LTS1}

Branching process in random environment (BPRE) is a basic and important branching system which characterizes the influence of external environment on branching behaviour. This model was first introduced by Smith and Wilkinson \cite{C1} for the independent and identically distributed environment case, and then generalized by  Athreya and Karlin \cite{C2} to the stationary and environment case. There were a lot of achievements on this topic in the literature; see for example \cite{B1,B2, glm17,L1, hl12, huang14, B3}. As an important  extension of BPRE, the model of   branching process with immigration (BPIRE) was proposed by adding the immigration on the basis of the branching process. In this model, based on the original branching mechanism, new immigrants come and join the family in each generation, and both reproduction and immigration distribution are determined by an environment varying with the time.
This means that the population size at every generation is influenced by both the environment and immigration populations.  In recent years,  many  results were obtained for BPIRE, including moments, limit theorems, convergence rates and so on; see for example \cite{ A2,D2,D1,wl17,T4}. In this paper, we are interested in the central limit theorem associated to the process. On this subject, for BPRE (without immigration), Huang and Liu \cite{hl12} showed a  central limit theorem on the logarithm of population size, Grama \emph{et al.} \cite{glm17} obtained the corresponding Berry-Esseen's bound,  and   Gao \cite{L1} presented the exact convergence rate; for BPIRE,  corresponding central limit theorem and  Berry-Esseen's bound were shown in Wang and Liu \cite{wl17, T4}. Based on their research, this paper aims to explore the exact convergence rate in central limit theorem for BPIRE and to reveal the influence of the immigration on it.

In this paper, we focus on a single-type branching process with immigration in a random  environment, which is defined by:
$$Z_{0}=1,\quad Z_{n+1}=Y_{n}+\sum_{i=1}^{Z_{n}}X_{n,i},n=0,1,2,\cdots, $$
where $Y_n$ is the number of new immigrants in the $n$-th generation, and $X_{n,i}$ is the number of offspring of the $i$-th individual in the $n$-th generation. The random environment, denoted by $\xi=(\xi_0,\xi_1,\xi_2,\cdots)$, is a sequence of independent and identically distributed (i.i.d.) random variables  taking values in abstract space $\Theta$. Each variable $\xi_n$ corresponds to two probability distribution on $\mathbb{N}=\{0,1,2,\cdots\}$, one is the offspring distribution denoted by
\begin{displaymath}
p(\xi_n)=\{p_k(\xi_n);k\geq0\},\quad \text{where $p_{k}(\xi_{n})\geq0$  and $\sum_{k}p_k(\xi_n)=1$;}
\end{displaymath}
the other is the immigration distribution denoted by
\begin{displaymath}
h(\xi_n)=\{h_k(\xi_n);k\geq0\},\quad \text{where $h_{k}(\xi_{n})\geq0$  and $\sum_{k}h_k(\xi_n)=1$.}
\end{displaymath}
Given the environment $\xi$, we suppose that  $X_{n,i}(n=0,1,2,\cdots,i=1,2,\cdots)$ and $Y_{n}(n\geq0)$ are all independent of each other, $X_{n,i}$ have the distribution $p(\xi_n)$ and $Y_{n}$ has the distribution ${h}(\xi_{n})$. In particular, if $Y_n=0$ for all $n$, there is no immigration and the process is a branching process in a random environment, denoted by $(\bar Z_n)$, which is defined by:
$$\bar Z_0=1,\qquad \bar Z_n=\sum_{i=1}^{\bar Z_{n}}X_{n,i}, n=0,1,2,\cdots $$

Let $(\Gamma,\mathbb{P}_{\xi})$ be the probability space under the given environment $\xi$. The probability $\mathbb{P}_{\xi}$ is the so-called quenched law. The total probability space
can be formulated as the product space $(\Gamma\times\Theta,\mathbb{P})$, where $\mathbb{P}(dx,d\xi)=\mathbb{P}_{\xi}(dx)\tau(d\xi)$ in the sense that for all measurable and positive function $g$, we have
$$\int g(x,\xi)\mathbb{P}(dx,d\xi)=\int\int g(x,\xi)\mathbb{P}_{\xi}(dx)\tau(d\xi)$$
where $\tau$ is the law of the environment $\xi$. The total probability $\mathbb{P}$ is usually called annealed law. The quenched law $\mathbb{P}_{\xi}$ may be considered to be the conditional probability of the annealed law $\mathbb P$ given $\xi$.

We write $X_n=X_{n,1}$ for brevity. For $n\in \mathbb N$, set
$$m_{n}=\sum_{k=0}^{\infty}k p_{k}(\xi_{n})=\mathbb E_\xi X_{n}.$$
Throughout the paper, we always assume that
\begin{equation}\label{ea0}
\mathbb P(X_0=0)=0\qquad\text{and}\qquad \mathbb P(X_0=1)<1,
\end{equation}
which means that each individual produces at least one child, and the probability of producing at least two children is positive. This assumption implies that $\mathbb P(m_0>1)=1$, so that $\mathbb E\log m_0>0$,
which means that the corresponding BPRE is supercritical. Therefore, the total population $Z_n$ tends to infinity almost surely (a.s.).  In order to investigate the asymptotic properties of  $Z_n$, we introduce the natural submartingale  and  martingale  of the model. For $n\in \mathbb N$, set
\begin{equation*}
\Pi_{0}=1,\quad \quad \Pi_{n}=\prod_{k=0}^{n-1}m_{k}, \;\;n=1,2,\cdots,
\end{equation*}
$$
W_{n}=\frac{Z_{n}}{\Pi_{n}}\qquad\text{and}\qquad
\bar W_{n}=\frac{\bar Z_{n}}{\Pi_{n}}.
$$
It is known that  $W_n$ forms a nonnegative submartingale, and it  converges almost surely (a.s.) to some limit $W$ if $\mathbb E \log m_0>0$ and $\mathbb E\log^+\frac{Y_0}{m_0}<\infty$ by \cite[Theorem 3.2]{wl17}, while $\bar W_n$ is a nonnegative martingale
 and hence it  converges a.s. to a limit $\bar W$ without any condition.

Among the limit properties associated to the process $Z_n$, a central limit theorem on $\log Z_n$  was established by Wang and Liu \cite[Theorem 7.1]{wl17}: if $\sigma^2=\mathbb E(\log m_0-\mu)^2\in(0,\infty)$, $\mathbb E\log^+\frac{Y_0}{m_0}<\infty$ and $\mathbb E\frac{X_0}{m_0}\log^+X_0<\infty$, then for $x\in\mathbb R$,
$$\lim_{n\to\infty}\mathbb P\left(\frac{\log Z_n -n\mu}{\sqrt{n}\sigma }\leq x\right)=\Phi(x), $$
where $\mu=\mathbb E\log m_0$ and $\Phi(x)$ is the standard normal distribution function. After that,  Wang and Liu  \cite[Theorem 1.1]{T4} further obtained the corresponding Berry-Essen bound which can describe the convergence rate of central limit theorem  under stronger moment conditions: if $\mathbb E(\log m_0)^{2+\delta}<\infty$ for some $\delta>0$, $\sigma>0$, $\mathbb E(\frac{Y_0}{m_0})^p<\infty$ and $\mathbb E(\frac{X_0}{m_0})^p<\infty$ for some $p>1$, then
$$\sup_{x\in\mathbb R}\left|\mathbb P\left(\frac{\log Z_n -n\mu}{\sqrt{n}\sigma }\leq x\right)-\Phi(x)\right|=O(n^{-\delta/2})\quad (n\to\infty).$$
These results generalized the results for BPRE showed in Huang and Liu \cite[Theorem 1.7]{hl12} and Grama \emph{et al.} \cite[Theorem 1.1]{glm17}. Recently, Gao \cite{L1} discovered the exact convergence rate of the central limit theorem for BPRE in the case that $\log m_0$ is a.s. non-lattice:  under certain moment conditions,
$$\mathbb P\left(\frac{\log Z_n -n\mu}{\sqrt{n}\sigma }\leq x\right)-\Phi(x)\sim \bar g(x)n^{-1/2}\quad (n\to\infty)$$
with a function $\bar g(x)$ that can be exactly represented. We will see that a similar result will also hold for BPIRE.

\begin{thm}\label{tt1}
 Assume that $\log m_0$ is a.s. non-lattice and $\mathbb E (\log m_0)^r<\infty$ for some $r\geq 3$. Set $\mu=\mathbb E\log m_0$, $\sigma^2=\mathbb E(\log m_0-\mu)^2$ and $\mu_3=\mathbb E(\log m_0-\mu)^3$. If $\sigma>0$, $\mathbb E(\frac{Y_0}{m_0})^\delta<\infty$ and $\mathbb E(\mathbb E_\xi (\frac{X_0}{m_0})^p)^\delta<\infty$ for some $p>1$ and  $\delta>0$,
then, for $x\in\mathbb R$,
\begin{equation}\label{te3}
\lim_{n\rightarrow\infty}\sqrt{n}\left[\mathbb P\left(\frac{\log Z_n -n\mu}{\sqrt{n}\sigma }\leq x\right)-\Phi(x)\right]=-\frac{1}{\sigma}\varphi(x)\mathbb E\log W+Q(x),
\end{equation}
where $\Phi(x)=\frac{1}{\sqrt 2\pi}\int_{-\infty}^x e^{-t^2/2}\mathrm{d}t$ is the standard normal distribution function, $\varphi(x)=\frac{1}{\sqrt 2\pi} e^{-x^2/2}$ is the density function of the standard normal distribution, and $Q(x)=\frac{1}{6\sigma^3}\mu_3(1-x^2)\varphi(x)$.
\end{thm}

Theorem \ref{tt1} reveals the exact convergence rate of the central limit theorem for BPIRE. Comparing it with the corresponding Berry-Essen bound shown by Wang and Liu \cite[Theorem 1.1]{T4}, we see that under ideal moment conditions,  it is possible to reach the maximum convergence rate $g(x)n^{-1/2}$ for central limit theorem, where $g(x)=-\frac{1}{\sigma}\varphi(x)\mathbb E\log W+Q(x)$. Moreover, it can also be seen that the moment conditions on $Y_0$ and $X_0$ in Theorem \ref{tt1} are weaker than those in \cite[Theorem 1.1]{T4} whenever $\mathbb E (\log m_0)^r<\infty$ for some $r\geq 3$, which means that Theorem \ref{tt1} is improved on the moment condition compared with \cite[Theorem 1.1]{T4}.  When there is no immigration,  the result of Theorem \ref{tt1} coincides with that obtained by Gao \cite{L1} for BPRE.

\medskip
The rest part of this paper is organised as follows. In Section \ref{s2}, we study  moments of $\log W_n$. Based on those moment results,    we  work on   convergence rates of $\log W_n$ in  Section  \ref{s3}. Finally,  Section  \ref{s4} is devoted to the proof of Theorem \ref{tt1}.

\section{Moments of $\log W_n$}\label{s2}
In this section, we will study the existence of the moments $\sup_n\mathbb E|\log W_n|^r$ for $r>0$. Noticing that $W_n\geq \bar W_n$, we will   consider the moments of $\bar W_n$, whose existence is related to the decay rates of the Laplace transform of the limit $\bar W$.
We   first  show a lemma which infers that a function $\phi(t)$ will be the bounded quantity of $(\log t)^{-r}$ when it meets appropriate conditions.
\begin{lem}\label{le0}
Let $\phi(t)$ be a bounded function satisfying
\begin{equation}\label{l1e1}
\phi(t)\leq q \mathbb E\phi(Gt)+\mathbb P(t<V)\qquad (\forall t\geq 0),
\end{equation}\label{l1e2}
where $q>0$ is a constant, and $G, V$ are non-negative random variables. If $q<1$, $\mathbb E|\log G|^r<\infty$ and $\mathbb E|\log V|^r<\infty$ for some $r>0$, then $\phi(t)=O((\log t)^{-r})$.
\end{lem}

\begin{proof}
We assume that   $|\phi(t)|$  is bounded by the constant $M>0$.
 Let $(G_n)_{n\geq 1}$ be the sequence of independent copies of $G$.
Using \eqref{l1e2} by iteration,  we can get that, for any  integer $l\geq 1$,
\begin{eqnarray}\label{l1e3}
\phi(t)&\leq &q^l\mathbb E\phi(G_l\cdots G_1 t)+\sum_{k=0}^{l-1}q^k\mathbb P(G_k\cdots G_1 G_0 t<V)\nonumber\\
&\leq & M q^l+\sum_{k=0}^{l-1}q^k\mathbb P(G_k\cdots G_1 G_0 t<V),
\end{eqnarray}
where we regard as $G_0=1$ by convention.
Since $\mathbb E|\log G|^r<\infty$ and $\mathbb E|\log V|^r<\infty$, by Markov's inequality, for $t>1$ and $k\geq 0$,
\begin{eqnarray}\label{l1e4}
\mathbb P(G_k\cdots G_1 G_0 t<V)&=& \mathbb P\left (\log t\leq -\sum_{j=1}^k \log G_j+\log V\right)\nonumber\\
&\leq & (\log t)^{-r}\mathbb E\left|-\sum_{j=1}^k \log G_j+\log V\right|^r\nonumber\\
&\leq & C (\log t)^{-r}\left[k^{\max\{r,1\}}\mathbb E|\log G|^r+\mathbb E|\log V|^r\right]\nonumber\\
&\leq& C (\log t)^{-r} (k+1)^{\max\{r,1\}},
\end{eqnarray}
where $C>0$ stands for a general constant, and it may differ from line to line. Since $q\in(0,1)$, we have $\sum_{k=0}^{\infty}q^k  (k+1)^{\max\{r,1\}}<\infty$.
Combining \eqref{l1e3} and \eqref{l1e4} yields that for $t>1$ and $l\geq 1$,
\begin{equation}\label{l1e5}
\phi(t)
\leq  M q^l+C (\log t)^{-r}\sum_{k=0}^{\infty}q^k  (k+1)^{\max\{r,1\}}= M q^l+C (\log t)^{-r}.
\end{equation}
Take $l=[-\frac{r\log\log t}{\log q}]$. Then $l\geq 1$ for $t$ large enough, and   $q^l\leq (\log t)^{-r}$.
So \eqref{l1e5} reveals that $\phi(t)=O((\log t)^{-r})$.
\end{proof}

Next, we consider $\mathbb E e^{-t\bar W}$ $(t>0)$, the Laplace transform of $\bar W$. With the help of Lemma \ref{le0}, we will see that it would be  the bounded quantity of $(\log t)^{-r}$ under appropriate moment conditions. This result is necessary in the proof of Theorem \ref{tt2}.

\begin{lem}\label{le1}
Let $\phi(t)=\mathbb E e^{-t\bar W}$ $(t>0)$ be the Laplace transform of $\bar W$.
If $\mathbb E(\log m_0)^{a}<\infty$ and $\mathbb E(\mathbb E_\xi (\frac{X_0}{m_0})^p)^\varepsilon<\infty$ for some $a>0$, $p>1$ and $\varepsilon>0$, then for $r\in (0,a)$,
$$\phi(t)=O((\log t)^{-r}).$$
\end{lem}

\begin{proof}The proof idea is inspired by \cite{hl12, GLM22+} by working on the quenched Laplace transform of $\bar W$ and constructing the recursive relationship.
We think  that $p\in(1,2]$, otherwise we use $\min\{p,2\}$ to  replace $p$. Set $\phi_\xi(t)=\mathbb E_\xi e^{-t \bar W}$.
Notice that the function $\frac{e^{-t}-1+t}{t^p}$ is bounded on $(0,\infty)$. Thus there exists a constant $C_p>0$ such that $e^{-t}\leq 1-t+ C_p t^p$ for all $t\geq 0$. For constant $K>0$, set $t_K=(KpC_p)^{-(p-1)^{-1}}$ and $\beta_K=1-t_K(1-1/p)\in(0,1)$. We can deduce that
\begin{equation}\label{l2e1}
\phi_\xi(t)\leq \phi_\xi(t_K)\leq 1-t+C_p t_K^p \mathbb E_\xi \bar W^p\leq 1-t_K+KC_p t_K^p=\beta_K
\end{equation}
if $t\geq t_K$ and $\mathbb E_\xi \bar W^p\leq K$.

 Le $T$ be the shift operator such that  $T\xi=(\xi_1,\xi_2,\cdots)$ if $\xi=(\xi_0,\xi_1,\cdots)$.
Observe that $\bar W=\frac{1}{m_0}\sum_{i=0}^{X_0}\bar W(i)$, where  given the environment $\xi$, $\bar W(i)$ $(i=1,2,\cdots)$ are i.i.d. with the same distribution $\mathbb P_\xi(\bar W(i)\in\cdot)=\mathbb P_{T\xi}(\bar W\in\cdot)$. Thus
$$\phi_\xi(t)=\mathbb E_\xi \left[\phi_{T\xi}(\frac{t}{m_0})^{X_0}\right]\leq \phi_{T\xi}(\frac{t}{m_0})\left(p_1(\xi_0)+(1-p_1(\xi_0))\phi_{T\xi}(\frac{t}{m_0})\right).$$
By iteration, for any integer $n\geq 1$,
$$\phi_\xi(t)\leq \phi_{T^n\xi}(\frac{t}{\Pi_n})\prod_{j=0}^{n-1}\left(p_1(\xi_j)+(1-p_1(\xi_j))\phi_{T^n\xi}(\frac{t}{\Pi_n})\right).$$
By \eqref{l2e1},
\begin{eqnarray*}
\phi_\xi(t)\leq \phi_{T^n\xi}(\frac{t}{\Pi_n})\prod_{j=0}^{n-1}\left(p_1(\xi_j)+(1-p_1(\xi_j))\beta_K\right)+\phi_{T^n\xi}(\frac{t}{\Pi_n})\mathbf{1}_{\{\mathbb E_{T^n\xi}\bar W^p>K\}}+\mathbf{1}_{\{t<V\}},
\end{eqnarray*}
where $V=t_K\Pi_n$. Taking the expectations yields that
\begin{eqnarray}\label{l2e2}
\phi(t)\leq \mathbb E\left[\phi(\frac{t}{\Pi_n})\prod_{j=0}^{n-1}\left(p_1(\xi_j)+(1-p_1(\xi_j))\beta_K\right)\right]+\mathbb E \left[\phi_{T^n\xi}(\frac{t}{\Pi_n})\mathbf{1}_{\{\mathbb E_{T^n\xi}\bar W^p>K\}}\right]+\mathbb P(t<V).
\end{eqnarray}
Applying the  Burholder's inequality, we see that
$$\mathbb E_\xi \bar W^p\leq 1+
\sum_{k=0}^\infty \Pi_k^{-(p-1)}\eta_k,$$
where $\eta_k=\mathbb E_{T^k\xi}|\bar W_1-1|^p$. Take
$\epsilon\in(0,\min\{\varepsilon,1\})$.
Thus,  with the notation $\Pi_{k,l}=\prod_{i=k}^{l-1}m_i$,
\begin{eqnarray}\label{l2e3}
\mathbb E \left[\phi_{T^n\xi}(\frac{t}{\Pi_n})\mathbf{1}_{\{\mathbb E_{T^n\xi}\bar W^p>K\}}\right]&\leq&K^{-\epsilon}\mathbb E \left[\phi_{T^n\xi}(\frac{t}{\Pi_n})(\mathbb E_{T^n\xi}\bar W^p)^\epsilon\right]\nonumber\\
&\leq&K^{-\epsilon}\mathbb E \left[\phi_{T^n\xi}(\frac{t}{\Pi_n})\left(1+\sum_{k=0}^\infty \Pi_{n,n+k}^{-(p-1)}\eta_{n+k}\right)^\epsilon\right]\nonumber\\
&\leq&K^{-\epsilon}\mathbb E \left[\phi_{T^n\xi}(\frac{t}{\Pi_n})+\sum_{k=0}^\infty \phi_{T^n\xi}(\frac{t}{\Pi_n})\Pi_{n,n+k}^{-(p-1)\epsilon}\eta_{n+k}^\epsilon\right]\nonumber\\
&\leq& K^{-\epsilon}\mathbb E \left[\phi_{T^n\xi}(\frac{t}{\Pi_n})+\sum_{k=0}^\infty \phi_{T^{n+k+1}\xi}(\frac{t}{\Pi_{n+k+1}})\Pi_{n,n+k}^{-(p-1)\epsilon}\eta_{n+k}^\epsilon\right]\nonumber\\
&=& K^{-\epsilon}\left[\mathbb E \phi(\frac{t}{\Pi_n})+\sum_{k=0}^\infty \mathbb E\phi(\frac{t}{\Pi_{n+k+1}})\Pi_{n,n+k}^{-(p-1)\epsilon}\eta_{n+k}^\epsilon\right].
\end{eqnarray}
Combining \eqref{l2e2} with \eqref{l2e3}, we obtain
\begin{eqnarray}\label{l2e4}
\phi(t)&\leq& \mathbb E\left[\phi(\frac{t}{\Pi_n})\prod_{j=0}^{n-1}\left(p_1(\xi_j)+(1-p_1(\xi_j))\beta_K\right)\right]\nonumber\\
&&+K^{-\epsilon}\left[\mathbb E \phi(\frac{t}{\Pi_n})+\sum_{k=0}^\infty \mathbb E\phi(\frac{t}{\Pi_{n+k+1}})\Pi_{n,n+k}^{-(p-1)\epsilon}\eta_{n+k}^\epsilon\right]+\mathbb P(t<V).
\end{eqnarray}
We define a non-negative random variable $G$ as follows: for all bounded measurable functions $g$,
\begin{eqnarray*}
\mathbb E g(G)&=&\frac{1}{q}\left\{
\mathbb E\left[g(\frac{1}{\Pi_n})\prod_{j=0}^{n-1}\left(p_1(\xi_j)+(1-p_1(\xi_j))\beta_K\right)\right]\right.\\
&&\left.+K^{-\epsilon}\left[\mathbb E g(\frac{1}{\Pi_n})+\sum_{k=0}^\infty \mathbb Eg(\frac{1}{\Pi_{n+k+1}})\Pi_{n,n+k}^{-(p-1)\epsilon}\eta_{n+k}^\epsilon\right]\right\},
\end{eqnarray*}
where
\begin{eqnarray*}
q=\left[\mathbb P(X_0=1)+(1-\mathbb P(X_0=1))\beta_K\right]^n+K^{-\epsilon}\left[1+\sum_{k=0}^\infty (\mathbb Em_0^{-(p-1)\epsilon})^k\mathbb E\eta_{0}^\epsilon\right]<\infty
\end{eqnarray*}
 is the norming constant.
Now \eqref{l2e4} becomes $\phi(t)\leq q\mathbb E\phi(Gt)+\mathbb P(t<V)$. Since
$$q\overset{n\uparrow\infty}{\longrightarrow}K^{-\epsilon}\left[1+\sum_{k=0}^\infty (\mathbb Em_0^{-(p-1)\epsilon})^k\mathbb E\eta_{0}^\epsilon\right] \overset{K\uparrow\infty}{\longrightarrow}0,$$
we can take $n, K$ large enough such that $q<1$. As $\mathbb E(\log m_0)^a<\infty$, we have for $r\in(0,a)$,
$$\mathbb E|\log V|^r=\mathbb E|\log t_K+\log \Pi_n|^r\leq (n+1)^{\max\{r-1,0\}}(\mathbb E|\log t_K|^r+\mathbb E(\log m_0)^r)<\infty,$$
and
\begin{eqnarray*}
q\mathbb E|\log G|^r&=&
\mathbb E\left[|-\log \Pi_n|^r\prod_{j=0}^{n-1}\left(p_1(\xi_j)+(1-p_1(\xi_j))\beta_K\right)\right]\\
&&+K^{-\epsilon}\left[
\mathbb E |-\log \Pi_n|^r+\sum_{k=0}^\infty \mathbb E |-\log \Pi_{n+k+1}|^r\Pi_{n,n+k}^{-(p-1)\epsilon}\eta_{n+k}^\epsilon
\right]\\
&\leq& n^{\max\{r,1\}}\mathbb E(\log m_0)^r(1+K^{-\epsilon})\\
&&+K^{-\epsilon}\sum_{k=0}^\infty(n+k+1)^{\max\{r-1,0\}}\sum_{j=0}^{n+k}\mathbb E\left[(\log m_j)^r\Pi_{n,n+k}^{-(p-1)\epsilon}\eta_{n+k}^\epsilon\right]\\
&=& n^{\max\{r,1\}}\mathbb E(\log m_0)^r(1+K^{-\epsilon})
+K^{-\epsilon}\sum_{k=0}^\infty(n+k+1)^{\max\{r-1,0\}}\\
&&\left\{\sum_{j=0}^{n-1}\mathbb E(\log m_0)^r (\mathbb Em_0^{-(p-1)\epsilon})^k\mathbb E\eta_0^\epsilon
+\sum_{j=n}^{n+k-1}\mathbb E[(\log m_0)^rm_0^{-(p-1)\epsilon}](\mathbb Em_0^{-(p-1)\epsilon})^{k-1}\mathbb E\eta_0^\epsilon\right.\\
&&\left.+(\mathbb Em_0^{-(p-1)\epsilon})^k\mathbb E[(\log m_0)^r\eta_0^\epsilon]\right\}\\
&<&\infty.
\end{eqnarray*}
It follows from Lemma \ref{le0} that
$\phi(t)=O((\log t)^{-r})$.
\end{proof}

The following theorem concerns the existence of the moments $\sup_n\mathbb E|\log W_n|^r$ for $r>0$.
\begin{thm}\label{tt2}
Let $a>0$. If $\mathbb E(\frac{Y_0}{m_0})^\delta<\infty$, $\mathbb E(\log m_0)^{a}<\infty$ and $\mathbb E(\mathbb E_\xi (\frac{X_0}{m_0})^p)^\varepsilon<\infty$ for some $\delta>0$, $p>1$ and $\varepsilon>0$, then we have  for $r\in(0, a)$, $$\sup_n\mathbb E|\log W_n|^r<\infty\quad\text{and}\quad \mathbb E|\log W|^r<\infty.$$
\end{thm}

\begin{proof}
Fix $r\in(0,a)$.
We just need to show that $\sup_n\mathbb E|\log W_n|^r<\infty$, which implies  $\mathbb E|\log W|^r<\infty$ by Fatou's lemma.
Without loss of generality,  we think that $\delta\in(0,1]$, otherwise we can use $\min\{\delta,1\}$ to replace $\delta$. Notice that
\begin{eqnarray*}
\mathbb E|\log W_n|^r=\mathbb E|\log W_n|^r\mathbf{1}_{\{W_n\geq 1\}}+\mathbb E|\log W_n|^r\mathbf{1}_{\{W_n< 1\}}
\leq C\mathbb E W_n^\delta+\mathbb E|\log W_n|^r\mathbf{1}_{\{W_n< 1\}}.
\end{eqnarray*}
Since $\mathbb E(\frac{Y_0}{m_0})^\delta<\infty$, by \cite[Theorem 6.3]{wl17}, we have $\sup_n\mathbb E W_n^\delta<\infty$. Then, set a function $g(x)$ as follows:
$$
g(x)=\left\{\begin{array}{ll}
(-\log x)^r & \text{if $0<x<x_r$;}\\
(-\log x_r)^r & \text{if $x\geq x_r$,}
\end{array}
\right.
$$
where $x_r=\min\{e^{r-1},1\}$. It can be seen that $g(x)$ is a decreasing and convex function on $(0,\infty)$ satisfying $g(x)\geq |\log x|^r\mathbf{1}_{\{x< 1\}}$ for $x>0$. Since $W_n\geq \bar W_n$, we have
\begin{equation}\label{p1e1}
\sup_n\mathbb E|\log W_n|^r\mathbf{1}_{\{W_n< 1\}}
\leq \sup_n\mathbb E g(W_n)\leq \sup_n\mathbb E g(\bar W_n)=\mathbb E g(\bar W),
\end{equation}
where we have used the monotonicity and convexity of $g(x)$ as well as \cite[lemma 2.1]{hl12}. Notice that
\begin{equation}\label{p1e2}
\mathbb E g(\bar W)\leq \mathbb E(-\log \bar W)^r\mathbf{1}_{\{\bar W< x_r\}}+(-\log x_r)^r.
\end{equation}
Combining \eqref{p1e1} and \eqref{p1e2}, we see that in order to obtain $\sup_n\mathbb E|\log W_n|^r\mathbf{1}_{\{W_n< 1\}}<\infty$, it remains to show that
\begin{equation}\label{ple31}
\mathbb E(-\log \bar W)^r\mathbf{1}_{\{\bar W< x_r\}}<\infty.
\end{equation}
We calculate that
\begin{eqnarray}\label{ple3}
\mathbb E(-\log \bar W)^r\mathbf{1}_{\{\bar W< x_r\}}
&\leq & r\int_{x_r^{-1}}^\infty t^{-1}(\log t)^{r-1}\mathbb P(\bar W<t^{-1})\mathrm{d}t\nonumber \\
&\leq & C \int_{x_r^{-1}}^\infty t^{-1}(\log t)^{r-1}\phi(t)\mathrm{d}t,
\end{eqnarray}
where $\phi(t)=\mathbb E e^{-t\bar W}$ is the Laplace transform of $\bar W$. By Lemma \ref{le1}, we have $\phi(t)=O((\log t)^{-r_1})$ for $r<r_1<a$.
Therefore, the integral in \eqref{ple3} is finite, which yields \eqref{ple31}.

\end{proof}

\section{Convergence rates of $\log W_n$}\label{s3}
In this section, we deduce the convergence rates of $\log W_n$ based the moment assumption that $\sup_n\mathbb E|\log W_n|^{r}<\infty$ for some  $r>0$.
We  will see that  $\log W_n$ can converge very fast, with a polynomial or  exponential rate,  under appropriate moment conditions.

\begin{lem}[\cite{L1}, Lemma 2.2]\label{l3l1}
Let $r>0$. Then $\mathbb E \bar Z_n^{-r}\leq (\mathbb E X_0^{-r})^n$, $\mathbb E(\log \bar Z_n)^{-r}\mathbf{1}_{\{Z_n\geq 2\}}\leq C n^{-r}$.
\end{lem}

\begin{thm}\label{tt3}
Let $\gamma>0$ and $q>0$. Assume that $\sup_n\mathbb E|\log W_n|^{r}<\infty$ for some  $r>q$, and
$\mathbb E(\frac{Y_0}{m_0})^\delta<\infty$ for some $\delta>0$.
\begin{itemize}
\item[(a)]If $\mathbb E\frac{X_0}{m_0}(\log^+ X_0)^\lambda<\infty$ for some $\lambda>\frac{r\gamma}{(r-q)\min\{\delta,q,1\}}$, then
$$\mathbb E|\log W_{n+1}-\log W_n|^q=O(n^{-\gamma}).$$
\item[(b)] If $\mathbb E(\mathbb E_\xi (\frac{X_0}{m_0})^p)^\delta<\infty$ for some $p>1$, then there exists a constant $\rho>1$ such that
$$\mathbb E|\log W_{n+1}-\log W_n|^q=O(\rho^{-n}).$$
\end{itemize}
\end{thm}
\begin{proof}
Without loss of generality, we  can think that $\delta\leq \min\{q,1\}$, otherwise we use $\min\{\delta,q,1\}$ to replace $\delta$.
Set $\Delta_n=\frac{W_{n+1}}{W_n}-1$. Then $\log W_{n+1}-\log W_n=\log (\Delta_n+1)$. Notice that for  $\epsilon\in(0,1]$,  there exists a $M\in(0,1)$ such that $|\log(x+1)|\leq C|x|^\epsilon$ on $[-M,\infty)$.  Take $\epsilon=\delta/q$. Then, by H\'older's inequality,
\begin{eqnarray*}
\mathbb E|\log W_{n+1}-\log W_n|^q &=&
\mathbb E |\log (\Delta_n+1)|^q\mathbf{1}_{\{\Delta_n\geq -M\}}+\mathbb E |\log (\Delta_n+1)|^q\mathbf{1}_{\{\Delta_n<-M\}}\\
&\leq& C\mathbb E|\Delta_n|^\delta+(\mathbb E|\log (\Delta_n+1)|^r)^{q/r}\mathbb P(|\Delta_n|>M)^{1/s}\\
&\leq&C\mathbb E|\Delta_n|^\delta+C(\mathbb E|\Delta_n|^\delta)^{1/s},
\end{eqnarray*}
where $s=\frac{r}{r-q}$.
Let
$$B_{n,i}=\frac{X_{n,i}}{m_n}-1\quad\text{and}\quad\bar{\Delta}_n=\frac{1}{Z_n}\sum_{i=1}^{Z_n }B_{n,i}.$$
We will write $B_0=B_{0,1}$ for brevity later.
Notice that
$$\Delta_n=\frac{W_{n+1}}{W_n}-1=\bar{\Delta}_n+\frac{Y_n}{Z_n m_n}.$$
Since $\mathbb E(\frac{Y_0}{m_0})^\delta<\infty$ and $\mathbb E(\frac{X_0}{m_0})^\delta\leq 1$,
by Lemma \ref{l3l1}, we have
\begin{eqnarray*}
\mathbb E|\Delta_n|^\delta&\leq &C\left(\mathbb E|\bar \Delta_n|^\delta+\mathbb E(\frac{Y_0}{m_0})^\delta \mathbb E Z_n^{-\delta}\right)\nonumber\\
&\leq& C\left[\mathbb E|\bar \Delta_n|^\delta+(\mathbb EX_0^{-\delta})^n\right].
\end{eqnarray*}
The assumption \eqref{ea0} implies that $\mathbb EX_0^{-\delta}<1$. Therefore, we just need to prove that
\begin{equation*}
\mathbb E|\bar \Delta_n|^\delta=O(a_n)
\end{equation*}
with $a_n=n^{-s\gamma}$ for the assertion (a), and $a_n=\rho^{-n}$ for the assertion (b), respectively.

We first work on the assertion (a). Notice that
\begin{eqnarray*}
\mathbb E|\bar \Delta_n|^\delta&=& \mathbb E|\bar \Delta_n|^\delta\mathbf{1}_{\{Z_n=1\}}+\mathbb E|\bar \Delta_n|^\delta\mathbf{1}_{\{Z_n\geq 2\}}\nonumber\\
&\leq&\mathbb E\left|\frac{X_0}{m_0}-1\right|\mathbb P(Z_n=1)+\mathbb E|\bar \Delta_n|^\delta\mathbf{1}_{\{Z_n\geq 2\}}
\end{eqnarray*}
and $\mathbb P(Z_n=1)=(\mathbb E[h_0(\xi_0)p_1(\xi_0)])^n
\leq\mathbb P(X_0=1)^n$.
It remains to show that
\begin{equation}\label{t3e1}
\mathbb E|\bar \Delta_n|^\delta\mathbf{1}_{\{Z_n\geq 2\}}\leq C n^{-s\gamma}.
\end{equation}
We use a technique of truncation. Let $I_n(x)=\mathbf{1}_{\{|x|\leq Z_n\}}$ be the truncation function and $I_n^c(x)=1-I_n(x)$. Put
$\tilde B_{n,i}=B_{n,i}I_n(B_{n,i})$. Then, we write
\begin{eqnarray*}
\bar \Delta_n&=&\frac{1}{Z_n}\sum_{i=1}^{Z_n}(B_{n,i}-
\tilde B_{n,i})+\frac{1}{Z_n}\sum_{i=1}^{Z_n}(\tilde B_{n,i}-
\mathbb E_n\tilde B_{n,i})+\frac{1}{Z_n}\sum_{i=1}^{Z_n}\mathbb E_n
\tilde B_{n,i}\\
&=&:\Lambda_{n,1}+\Lambda_{n,2}+\Lambda_{n,3},
\end{eqnarray*}
where the notation $\mathbb E_n(\cdot)=\mathbb E[\cdot|Z_n]$ is the conditional expectation respect to $Z_n$.
In order to obtain \eqref{t3e1}, we need to show that $\mathbb E|\Lambda_{n,i}|^\delta
\mathbf{1}_{\{Z_n\geq 2\}}\leq C n^{-s\gamma}$, $i=1,2,3$.

First, we consider $\Lambda_{n,3}$. Since $\mathbb E_n B_{n,i}=0$, we can write $\mathbb E_n\tilde B_{n,i}=\mathbb E_n(\tilde B_{n,i}-B_{n,i})=\mathbb E_nB_{n,i}I_n^c(B_{n,i})$. Notice that the function $(\log x)^{s\gamma/\delta}$ is increasing on $[1,\infty)$. By Lemma \ref{l3l1}, we have
\begin{eqnarray*}
&&\mathbb E |\Lambda_{n,3}|^\delta
\mathbf{1}_{\{Z_n\geq 2\}}\\
&\leq&\mathbb E\mathbf{1}_{\{Z_n\geq 2\}} Z_n^{-\delta}\left|\sum_{i=1}^{Z_n}\mathbb EB_{n,i}I_n^c(B_{n,i})\right|^\delta\\
&\leq&\mathbb E\mathbf{1}_{\{Z_n\geq 2\}} Z_n^{-\delta}\left(\sum_{i=1}^{Z_n}\mathbb E|B_{n,i}|\frac{(\log^+|B_{n,i}|)^{s\gamma/\delta}}{(\log Z_n)^{s\gamma/\delta}}\right)^\delta\\
&\leq&\mathbb E\mathbf{1}_{\{Z_n\geq 2\}}(\log Z_n)^{-s\gamma}\left(\mathbb E |B_0|(\log^+|B_0|)^{s\gamma/\delta}\right)^\delta\\
&\leq& Cn^{-s\gamma},
\end{eqnarray*}
since $\mathbb E\frac{X_0}{m_0}(\log^+X_0)^\lambda<\infty$ and $s\gamma/\delta\leq \frac{r\gamma}{(r-q)\min\{\delta,q,1\}}<\lambda$.

Second, we deal with $\Lambda_{n,1}$. Noticing that the function $x^{1-\delta}(\log x)^{s\gamma}$ is increasing on $[1,\infty)$, again by Lemma \ref{l3l1} , we get
 \begin{eqnarray*}
&&\mathbb E |\Lambda_{n,1}|^\delta
\mathbf{1}_{\{Z_n\geq 2\}}\\
&\leq&\mathbb E\mathbf{1}_{\{Z_n\geq 2\}}Z_n^{-\delta}\sum_{i=1}^{Z_n}|B_{n,i}|^\delta\frac{|B_{n,i}|^{1-\delta}(\log^+|B_{n,i}|)^{s\gamma}}{Z_n^{1-\delta}(\log Z_n)^{s\gamma}}\\&\leq&\mathbb E\mathbf{1}_{\{Z_n\geq 2\}}(\log Z_n)^{-s\gamma}\mathbb E |B_0|(\log^+|B_0|)^{s\gamma}\\
&\leq& Cn^{-s\gamma},
\end{eqnarray*}
since  $s\gamma<\lambda$.

Finally, we work on $\Lambda_{n,2}$. Take $r_1>1$. Using Jensen's inequality and Burkholder's inequality, we get
\begin{eqnarray}\label{t3e2}
\mathbb E |\Lambda_{n,2}|^\delta
\mathbf{1}_{\{Z_n\geq 2\}}
&\leq&\mathbb E\mathbf{1}_{\{Z_n\geq 2\}}\left(\mathbb E_n\left|\frac{1}{Z_n}\sum_{i=1}^{Z_n}(\tilde B_{n,i}-
\mathbb E_n\tilde B_{n,i})\right|^{r_1}\right)^{\delta/r_1}\nonumber\\
&\leq&C\mathbb E\mathbf{1}_{\{Z_n\geq 2\}}Z_n^{-\delta}\left(\sum_{i=1}^{Z_n}\mathbb E_n|\tilde B_{n,i}-
\mathbb E_n\tilde B_{n,i}|^{r_1}\right)^{\delta/r_1}\nonumber\\
&\leq&C\mathbb E\mathbf{1}_{\{Z_n\geq 2\}}Z_n^{-\delta}\left(\sum_{i=1}^{Z_n}\mathbb E_n|B_{n,i}|^{r_1}I_n(B_{n,i})\right)^{\delta/r_1}.
\end{eqnarray}
Notice that the function $x^{r_1-1}(\log x)^{-s\gamma r_1/\delta}$ is positive and increasing on $[c,\infty)$ for some constant $c>0$ large enough. With this $c$, considering \eqref{t3e1},  we can obtain
\begin{eqnarray}\label{t3e3}
\mathbb E |\Lambda_{n,2}|^\delta
\mathbf{1}_{\{Z_n\geq 2\}}
&\leq& C\left\{\mathbb E\mathbf{1}_{\{Z_n\geq 2\}}Z_n^{-\delta}\left(\sum_{i=1}^{Z_n}\mathbb E_n|B_{n,i}|^{r_1}I_n(B_{n,i})\mathbf{1}_{\{|B_{n,i}|\geq c\}}\right)^{\delta/r_1}\right.\nonumber\\
&&+\left.\mathbb E\mathbf{1}_{\{Z_n\geq 2\}}Z_n^{-\delta}\left(\sum_{i=1}^{Z_n}\mathbb E_n|B_{n,i}|^{r_1}I_n(B_{n,i})\mathbf{1}_{\{|B_{n,i}|<c\}}\right)^{\delta/r_1}\right\}\nonumber\\
&=&:C\left(\Lambda_{n,2}^{(1)}+\Lambda_{n,2}^{(2)}\right).
\end{eqnarray}
It is not difficult to see that
\begin{eqnarray}\label{t3e4}
\Lambda_{n,2}^{(2)}\leq C \mathbb E\mathbf{1}_{\{Z_n\geq 2\}}Z_n^{-(1- {1}/{r_1})\delta}\leq C\left(\mathbb E X_0^{-(1-1/r_1)\delta}\right)^n.
\end{eqnarray}
For $\Lambda_{n,2}^{(1)}$, by the increasing monotonicity of the function $x^{r_1-1}(\log x)^{-s\gamma r_1/\delta}$ on $[c,\infty)$, we have
\begin{eqnarray}\label{t3e5}
\Lambda_{n,2}^{(1)}&\leq &C \mathbb E\mathbf{1}_{\{Z_n\geq 2\}}Z_n^{-\delta}\left(\sum_{i=1}^{Z_n}\mathbb E_n|B_{n,i}|^{r_1}\frac{Z_n^{r_1-1}(\log^+|B_{n,i}|)^{s\gamma r_1/\delta}}{|B_{n,i}|^{r_1-1}(\log Z_n)^{s\gamma r_1/\delta}}\right)^{\delta/r_1}\nonumber\\
&\leq& C\mathbb E\mathbf{1}_{\{Z_n\geq 2\}}(\log Z_n)^{-s\gamma}\left(\mathbb E |B_0|(\log^+|B_0|)^{s\gamma r_1/\delta}\right)^{\delta/r_1}\nonumber\\
&\leq& Cn^{-s\gamma}
\end{eqnarray}
if we take $r_1>1$ small enough such that $s\gamma r_1/\delta<\lambda$. Combining \eqref{t3e3} with  \eqref{t3e4} and  \eqref{t3e5} yields that  $\mathbb E |\Lambda_{n,2}|^\delta\mathbf{1}_{\{Z_n\geq 2\}}\leq  Cn^{-s\gamma}$ if $\lambda>s\gamma /\delta$.

Now we work on the assertion (b). We think that $p\in(1,2]$. Denote by $\mathbb E_{\xi,n}(\cdot)$ the conditional expectation when  $\xi$ and $Z_n$ are given. By Burkholder's inequality,
\begin{eqnarray*}
\mathbb E_{\xi,n}|\bar \Delta_n|^\delta&\leq& \left(\mathbb E_{\xi,n}|\bar \Delta_n|^p\right)^{\delta/p}\\
&\leq& C Z_n^{-\delta}\left(\mathbb E_{\xi,n}\sum_{i=1}^{Z_n}|B_{n,i}|^p\right)^{\delta/p}\\
&=& C Z_n^{-\delta(1-1/p)}\left(\mathbb E_{T^n\xi}|B_0|^p\right)^{\delta/p}.
\end{eqnarray*}
Taking the expectation gives
\begin{equation}\label{ea1}
\mathbb E |\bar \Delta_n|^\delta\leq C \mathbb EZ_n^{-\delta(1-1/p)} \mathbb E \left(\mathbb E_{\xi}|B_0|^p\right)^{\delta/p}\leq C \left(\mathbb E X_0^{-\delta(1-1/p)}\right)^n.
\end{equation}
We can take $\rho=\left(\mathbb E X_0^{-\delta(1-1/p)}\right)^{-1}>1$. The proof is complete.
\end{proof}

\begin{re}Theorem \ref{tt3} is established on the  moment assumption $\sup_n\mathbb E|\log W_n|^{r}<\infty$ to describe the $L^q$ ($q>0$) convergence rates of $\log W_n$. Theorem \ref{tt3}(a) reveals a
 polynomial rate, while   Theorem \ref{tt3}(b) shows an exponential rate. Obviously, the exponential rate is also polynomial, thus both conditions in (a) and (b) of Theorem \ref{tt3} will lead to the polynomial rate. However, in order to ensure $\sup_n\mathbb E|\log W_n|^{r}<\infty$,  by Theorem \ref{tt2}, we need the moment condition $\mathbb E(\mathbb E_\xi (\frac{X_0}{m_0})^p)^\varepsilon<\infty$ for some $ p>1$ and $\varepsilon>0$ which is also the condition of  Theorem \ref{tt3}(b). It implies that if we use Theorem \ref{tt2} to derive $\sup_n\mathbb E|\log W_n|^{r}<\infty$,  we also deduce that $\log W_n$ converges in $L^q$ with an exponential rate at the same time by Theorem \ref{tt3}(b). Thus,   Theorem \ref{tt3}(a)  has no sense in this paper and we will not use it later. If one can find a condition weaker than that of Theorem \ref{tt3}(b), then we can use Theorem \ref{tt3}(a) to obtain another available condition to achieve the  polynomial  convergence rate of $\log W_n$.
\end{re}

Combining Theorems \ref{tt2}  and \ref{tt3}, we deduce the following corollary about the $L^1$- convergence of $\log W_n$.

\begin{co}\label{co1}
If $\mathbb E(\frac{Y_0}{m_0})^\delta<\infty$, $\mathbb E(\log m_0)^{r}<\infty$ and $\mathbb E(\mathbb E_\xi (\frac{X_0}{m_0})^p)^\delta<\infty$ for some $r, p>1$ and $\delta>0$,    then
$$\lim_{n\rightarrow}\mathbb E\log W_{n}=\mathbb E\log W.$$
\end{co}
\begin{proof}By Theorem  \ref{tt2},  we have $\sup_n\mathbb E|\log W_n|^{r_1}<\infty$ for $r_1\in(1, r)$. Then applying Theorem \ref{tt3}(b) to $q=1$.
\end{proof}

With the help of  Theorems \ref{tt2}  and \ref{tt3}(b), by following the proof of  Wang and Liu \cite{T4},  we can weaken the condition in the Berry-Esseen bound on $\log Z_n$ established by Wang and Liu \cite{T4}. Compared with \cite {T4}, we see that the moment restriction on the immigration $\mathbb E(\frac{Y_0}{m_0})^p<\infty$ for some $p>1$ can be relaxed to   $p>0$, and the moment restriction on the reproduction $\mathbb E(\frac{X_0}{m_0})^p<\infty$ can be weaken to $\mathbb E(\mathbb E_\xi (\frac{X_0}{m_0})^p)^\delta<\infty$ with some $\delta>0$.

\begin{thm}[Berry-Esseen bound]\label{ttb}
Assume that $\mathbb E(\log m_0)^{2+\epsilon}$ for some $\epsilon\in(0,1]$. Set $\mu=\mathbb E\log m_0$ and $\sigma^2=\mathbb E(\log m_0-\mu)^2$.
If $\sigma>0$, $\mathbb E(\frac{Y_0}{m_0})^\delta<\infty$ and $\mathbb E(\mathbb E_\xi (\frac{X_0}{m_0})^p)^\delta<\infty$ for some $p>1$  and   $\delta>0$,
 then
$$\sup_{x\in\mathbb R}\left|\mathbb P\left(\frac{\log Z_n -n\mu}{\sqrt{n}\sigma }\leq x\right)-\Phi(x)\right|\leq Cn^{-\epsilon/2}.$$
\end{thm}

\section{Proof of Theorem \ref{tt1}}\label{s4}
In this section, we go to the proof of Theorem \ref{tt1}. 
Notice that
$$\log Z_n=S_n+\log W_n,$$
where $S_n=\log \Pi_n$ is the  partial sum of the i.i.d. sequence $(\log m_n)_{n\geq 0}$. According to the classical knowledge of probability theory, it holds for  $S_n$ the Berry-Essen bound and the exact convergence rate in central limit theorem. Therefore, one can expect that $\log Z_n$ possesses the same  asymptotic properties as $S_n$ when $\log W_n$ converges fast enough.

For a random walk $(S_n)$, Lemma \ref{s4l1} below gives the corresponding Berry-Esseen  bound   in the part (a) and the exact convergence rate  in the part (b).
\begin{lem}[\cite{L2}]\label{s4l1}
Let $X$ be a random variable and $X_n$ be independence copies of $X$. Set $S_n=\sum_{k=1}^nX_k$, $\mu=\mathbb EX$, $\sigma^2=\mathbb E(X-\mu)^2$ and $\mu_3=\mathbb E(X-\mu)^3$. Assume that $\sigma>0$. Then,
\begin{itemize}
\item[(a)]\emph{(Berry-Esseen bound)} If $\mathbb E|X|^{2+\epsilon}<\infty$ for some $\epsilon\in(0,1]$,
then
$$\sup_{x\in\mathbb R}\left|\mathbb P\left(\frac{S_n -n\mu}{\sqrt{n}\sigma }\leq x\right)-\Phi(x)\right|\leq Cn^{-\epsilon/2};$$
\item[(b)]\emph{(Convergence rate)}
If $X$ is non-lattice and $\mathbb E|X|^3<\infty$, then
$$\lim_{n\rightarrow\infty}\sqrt{n}\sup_{x\in\mathbb R}\left\{\left|\mathbb P\left(\frac{S_n -n\mu}{\sqrt{n}\sigma }\leq x\right)-\Phi(x)-\frac{1}{\sqrt{n}}Q(x)\right|\right\}=0,$$
\end{itemize}
where $\Phi(x)=\frac{1}{\sqrt 2\pi}\int_{-\infty}^x e^{-t^2/2}\mathrm{d}t$ is   the standard normal distribution function, $\varphi(x)=\frac{1}{\sqrt 2\pi} e^{-x^2/2}$ is the density function of the standard normal distribution, and $Q(x)=\frac{1}{6\sigma^3}\mu_3(1-x^2)\varphi(x)$.
\end{lem}
Now we give the proof of Theorem \ref{tt1}.
\begin{proof}[Proof of Theorem \ref{tt1}]
Set $\alpha_n=n^{-\alpha}$ and $k_n=[n^\beta]$, where $\alpha>1/2$ and $\beta\in(0,1/2)$ are constants to be determined later. Let
$$D_n=\frac{\log W_n-\log W_{k_n}}{\sqrt{n}\sigma}\quad\text{and}\quad S_n=\log \Pi_n=\sum_{k=0}^{n-1}\log m_k.$$
Observe that
\begin{equation}\label{eo}
\mathbb P\left(\frac{\log Z_n -n\mu}{\sqrt{n}\sigma }\leq x\right)\left\{
\begin{array}{l}
\leq \mathbb P\left(\frac{\log Z_{k_n} -k_n\mu}{\sqrt{n}\sigma }+\frac{S_n-S_{k_n} -(n-k_n)\mu}{\sqrt{n}\sigma }\leq x+\alpha_n\right)+\mathbb P(|D_n|>\alpha_n)\\
\geq \mathbb P\left(\frac{\log Z_{k_n} -k_n\mu}{\sqrt{n}\sigma }+\frac{S_n-S_{k_n} -(n-k_n)\mu}{\sqrt{n}\sigma }\leq x-\alpha_n\right)-\mathbb P(|D_n|>\alpha_n)
\end{array}
\right.
\end{equation}
It suffices to prove that
\begin{equation}\label{t4ee1}
\sqrt{n}\mathbb P(|D_n|>\alpha_n)\rightarrow 0
\end{equation}
and
\begin{eqnarray}\label{t4ee2}
&&\mathbb P\left(\frac{\log Z_{k_n} -k_n\mu}{\sqrt{n}\sigma }+\frac{S_n-S_{k_n} -(n-k_n)\mu}{\sqrt{n}\sigma }\leq x\pm \alpha_n\right)\nonumber\\&=&\Phi(x)-\frac{1}{\sigma\sqrt{n}}\varphi(x)\mathbb E\log W+\frac{1}{\sqrt{n}}Q(x)+\frac{1}{\sqrt{n}}o(1).
\end{eqnarray}

We first prove \eqref{t4ee1}. By Theorem  \ref{tt2},  we have $\sup_n\mathbb E|\log W_n|^{r_1}<\infty$ for $r_1\in(1, r)$.
By Markov's inequality and applying Theorem  \ref{tt3}(b),
\begin{eqnarray*}
\sqrt{n}\mathbb P(|D_n|>\alpha_n)&\leq&\sqrt{n}\alpha_n^{-q}{\mathbb E|D_n|^q}\nonumber\\
&=&\sigma^{-1}n^{\alpha q+\frac{1}{2}(1-q)}\mathbb E|\log W_n-\log W_{k_n}|^q\nonumber\\
&\leq&\sigma^{-1}n^{\alpha q+\frac{1}{2}(1-q)}\sum_{k=k_n}^{n-1}\mathbb E|\log W_{k+1}-\log W_{k}|^q\nonumber\\
&\leq&Cn^{\alpha q+\frac{1}{2}(1-q)}\sum_{k=k_n}^{n-1}\rho^{-k}\nonumber\\
\nonumber\\
&\leq& Cn^{\alpha q+\frac{1}{2}(1-q)}\rho^{-k_n}\overset{n\rightarrow\infty}\longrightarrow 0.
\end{eqnarray*}
Now let us prove \eqref{t4ee2}. We just consider the case that $x+\alpha_n$. The case that $x-\alpha_n$ can be proved with $-\alpha_n$ in place of $\alpha_n$. Denote $F_n(x)=\mathbb P(\frac{S_n-n\mu}{\sqrt{n}\sigma}\leq x)$. Decompose
\begin{eqnarray*}
&&\mathbb P\left(\frac{\log Z_{k_n} -k_n\mu}{\sqrt{n}\sigma }+\frac{S_n-S_{k_n} -(n-k_n)\mu}{\sqrt{n}\sigma }\leq x\pm \alpha_n\right)\nonumber\\
&=&\int \mathbb P\left(\left.\frac{y -k_n\mu}{\sqrt{n}\sigma }+\frac{S_n-S_{k_n} -(n-k_n)\mu}{\sqrt{n}\sigma }\leq x\pm \alpha_n\right|\log Z_{k_n}=y\right)\mathbb P(\log Z_{k_n}\leq\mathrm{d}y )\nonumber\\
&=&\int F_{n-k_n}(y_n)\mathbb P(\log Z_{k_n}\leq\mathrm{d}y)\nonumber\\
&=&\int \left[ F_{n-k_n}(y_n)-\Phi(y_n)-\frac{1}{\sqrt{n-k_n}}Q(y_n)\right]\mathbb P(\log Z_{k_n}\leq\mathrm{d}y)\nonumber\\&&+\int \Phi(y_n) \mathbb P(\log Z_{k_n}\leq\mathrm{d}y)+\frac{1}{\sqrt{n-k_n}}\int  Q(y_n) \mathbb P(\log Z_{k_n}\leq\mathrm{d}y)
\nonumber\\
&=&: I_{n,1}+I_{n,2}+I_{n,3},
\end{eqnarray*}
where $y_n=\frac{\sqrt{n}}{\sqrt{n-k_n}}(x+\alpha_n)-\frac{y-k_n\mu}{\sqrt{n-k_n}\sigma}$. It follows that
\begin{eqnarray*}
&&\sqrt{n}\left[\mathbb P\left(\frac{\log Z_{k_n} -k_n\mu}{\sqrt{n}\sigma }+\frac{S_n-S_{k_n} -(n-k_n)\mu}{\sqrt{n}\sigma }\leq x\pm \alpha_n\right)-\Phi(x)\right]\nonumber\\
&=&\sqrt{n}I_{n,1}+\sqrt{n}\left(I_{n,2}-\Phi(x)\right)+\sqrt{n}I_{n,3}.
\end{eqnarray*}
We shall prove that the following three assertions:
\begin{eqnarray}
&&\sqrt{n}I_{n,1}=o(1);\label{et401}\\
&&\sqrt{n}(I_{n,2}-\Phi(x))=-\frac{1}{\sigma}\varphi(x)\mathbb E\log W+o(1);\label{et402}\\
&&\sqrt{n}I_{n,3}=Q(x)+o(1).\label{et403}
\end{eqnarray}

For $I_{n,1}$, by the Berry-Esseen bound for i.i.d sequence (see Lemma \ref{s4l1}(a)) and the boundedness of the function $Q(x)$, we have
$$ \sqrt{n-k_n}\sup_{x\in\mathbb R}\left|F_{n-k_n}(x)-\Phi(x)-\frac{1}{\sqrt{n-k_n}}Q(x)\right|\leq C.$$
Applying Lemma \ref{s4l1}(b) and the dominated convergence theorem, we get
$$\sqrt{n}|I_{n,1}|\leq \frac{\sqrt{n}}{\sqrt{n-k_n}}\int \sup_{x\in\mathbb R}\left|F_{n-k_n}(x)-\Phi(x)-\frac{1}{\sqrt{n-k_n}}Q(x)\right|\mathbb P(\log Z_{k_n}\leq\mathrm{d}y)\overset{n\rightarrow\infty}\longrightarrow 0,$$
which implies \eqref {et401}.

 Denote
$$U_n=\frac{\log Z_n -n\mu}{\sqrt{n}\sigma}\quad \text{and}\quad x_n=\frac{\sqrt{n}}{\sqrt{n-k_n}}(x+\alpha_n).$$
 By the central limit theorem for $\log Z_n$ (see \cite[Theorem 7.1]{wl17}), $U_n$ converges to $U\sim\mathcal{N}(0,1)$ in law. As $\mathbb E(\log m_0)^r<\infty$, we can calculate that $\sup_n\mathbb E|U_n|^r<\infty$. Thus, for $0<s<r$ the sequence $(U_n^s)$ is uniformly integrable. Consequently (noticing that $r>2$), we have
\begin{equation}\label{et4e3}
\lim_{n\rightarrow\infty} E|U_n|=E|U|=\sqrt{\frac{2}{\pi}},\qquad
\lim_{n\rightarrow\infty} EU_n^2=EU^2=1.
\end{equation}
By Taylor's expansion,
\begin{eqnarray}\label{et4e5}
&&\sqrt{n}(I_{n,2}-\Phi(x))\nonumber\\&=&\sqrt{n}\varphi(x)\mathbb E\left[x_n-x-\frac{\sqrt{k_n}}{\sqrt{n-k_n}}U_{k_n}\right]+\frac{1}{2}\sqrt{n}\mathbb E\left[\varphi'(\zeta_{n,x})\left(x_n-x-\frac{\sqrt{k_n}}{\sqrt{n-k_n}}U_{k_n}\right)^2\right],\end{eqnarray}
where $\zeta_{n,x}$ is a random  variable between $x$ and $x_n-\frac{\sqrt{k_n}}{\sqrt{n-k_n}}U_{k_n}$.
We calculate that
\begin{equation}\label{et4e4}
x_n-x=O(n^{\max\{\beta-1, -\alpha\}}).
\end{equation}
Noticing that $\mathbb EU_n=\frac{1}{\sqrt{n}\sigma}\mathbb E\log W_n$, by \eqref{et4e4} and Corollary \ref{co1}, we deduce that
\begin{eqnarray}\label{et4e6}
&&\sqrt{n}\varphi(x)\mathbb E\left[x_n-x-\frac{\sqrt{k_n}}{\sqrt{n-k_n}}U_{k_n}\right]\nonumber\\&=&\varphi(x)\left[\sqrt{n}(x_n-x)-\frac{\sqrt{n}}{\sigma\sqrt{n-k_n}}\mathbb E\log W_{k_n}\right]
\overset{n\rightarrow\infty}\longrightarrow -\frac{1}{\sigma}\varphi(x)\mathbb E\log W.
\end{eqnarray}
Since $\sup_{x\in\mathbb R}|\varphi'(x)|\leq C$, by using \eqref{et4e4} and \eqref{et4e3}, we obtain
\begin{eqnarray}\label{et4e7}
&&\frac{1}{2}\sqrt{n}\mathbb E\left|\varphi'(\zeta_{n,x})\left(x_n-x-\frac{\sqrt{k_n}}{\sqrt{n-k_n}}U_{k_n}\right)^2\right|\nonumber\\&\leq&C\left[\sqrt{n}(x_n-x)^2+\sqrt{n}\frac{k_n}{n-k_n}\mathbb EU_{k_n}^2\right]
\overset{n\rightarrow\infty}\longrightarrow 0.
\end{eqnarray}
Combining \eqref{et4e5} with \eqref{et4e6} and \eqref{et4e7} yields \eqref{et402}.

Finally, we consider $I_{n,3}$, and prove \eqref{et403}. By Taylor's expansion,
$$\sqrt{n}I_{n,3}=\frac{\sqrt{n}}{\sqrt{n-k_n}}Q(x)+\frac{\sqrt{n}}{\sqrt{n-k_n}}\mathbb E\left[Q'(\tilde \zeta_{n,x})(x_n-x-\frac{\sqrt{k_n}}{\sqrt{n-k_n}}U_{k_n})\right],
$$
where $\tilde\zeta_{n,x}$ is a random  variable between $x$ and $x_n-\frac{\sqrt{k_n}}{\sqrt{n-k_n}}U_{k_n}$. The proof will be finished if we prove that
$$\mathbb E\left[Q'(\tilde \zeta_{n,x})(x_n-x-\frac{\sqrt{k_n}}{\sqrt{n-k_n}}U_{k_n})\right]=o(1).$$
In fact, since $\sup_{x\in\mathbb R}|Q'(x)|\leq C$, by \eqref{et4e4} and \eqref{et4e3},
\begin{eqnarray*}\label{et4e8}
\mathbb E\left|Q'(\tilde\zeta_{n,x})\left(x_n-x-\frac{\sqrt{k_n}}{\sqrt{n-k_n}}U_{k_n}\right)\right| \leq  C\left(|x_n-x|+\frac{\sqrt{k_n}}{\sqrt{n-k_n}}\mathbb E|U_{k_n}|\right)
\overset{n\rightarrow\infty}\longrightarrow 0,
\end{eqnarray*}
which completes the proof.
\end{proof}

\end{document}